\documentclass[11pt,a4paper]{amsart}
\usepackage{amssymb,amsmath,amsthm,enumerate,mathtools,float,pinlabel,color}
\usepackage{lscape,tabularx,makecell,multirow,stmaryrd}
\newtheorem{theorem}{Theorem}[section]
\newtheorem{prop}[theorem]{Proposition}
\newtheorem{lemma}[theorem]{Lemma}
\newtheorem{cor}[theorem]{Corollary}
\newtheorem{conj}[theorem]{Conjecture}
\newtheorem{cor*}{Corollary}
\newtheorem{prop*}{Proposition}
\theoremstyle{definition}
\newtheorem{definition}[theorem]{Definition}
\newtheorem{defn}[theorem]{Definition}
\newtheorem{rem}[theorem]{Remark}

\newcommand{\Mod}{\mathrm{Mod}}
\newcommand{\Aut}{\mathrm{Aut}_k}

\newcommand{\LMod}{\mathrm{LMod}}
\newcommand{\SMod}{\mathrm{SMod}}
\newcommand{\Homeo}{\mathrm{Homeo}}

\newcommand{\lcm}{\mathrm{lcm}}
\newcommand{\lb}{\llbracket}
\newcommand{\rb}{\rrbracket}
\renewcommand{\P}{\mathcal{P}}
\newcommand{\Z}{\mathbb{Z}}
\renewcommand{\O}{\mathcal{O}}
\newcommand{\F}{\mathcal{F}}
\newcommand{\G}{\mathcal{G}}
\newcommand{\D}{\mathcal{D}}

\newcommand{\orb}{\mathrm{orb}}
\newcommand{\M}{\mathcal{M}}
\newcommand{\Dic}{\mathrm{Dic}}

\makeatletter
\@namedef{subjclassname@2020}{\textup{2020} Mathematics Subject Classification}
\makeatother

\begin{document}
\title[Metacyclic actions on surfaces]{Metacyclic actions on surfaces}

\author[K. Rajeevsarathy]{Kashyap Rajeevsarathy}
\address{Department of Mathematics\\
	Indian Institute of Science Education and Research Bhopal\\
	Bhopal Bypass Road, Bhauri \\
	Bhopal 462 066, Madhya Pradesh\\
	India}
\email{kashyap@iiserb.ac.in}
\urladdr{https://home.iiserb.ac.in/$_{\widetilde{\phantom{n}}}$kashyap/}

\author[A. Sanghi]{Apeksha Sanghi}
\address{Department of Mathematics\\
	Indian Institute of Science Education and Research Bhopal\\
	Bhopal Bypass Road, Bhauri \\
	Bhopal 462 066, Madhya Pradesh\\
	India}
\email{apeksha16@iiserb.ac.in}

\subjclass[2020]{Primary 57K20, Secondary 57M60}

\keywords{surface; mapping class; finite order maps; metacyclic subgroups}

\maketitle

\begin{abstract}
	Let $\mathrm{Mod}(S_g)$ be the mapping class group of the closed orientable surface $S_g$ of genus $g\geq 2$. In this paper, we derive necessary and sufficient conditions under which two torsion elements in $\mathrm{Mod}(S_g)$ will have conjugates that generate a finite metacyclic subgroup of $\mathrm{Mod}(S_g)$. This yields a complete solution to the problem of liftability of periodic mapping classes under finite cyclic covers. As applications of the main result, we show that $4g$ is a realizable upper bound on the order of a non-split metacyclic action on $S_g$ and this bound is realized by the action of a dicyclic group. Moreover, we give a complete characterization of the dicyclic subgroups of $\mathrm{Mod}(S_g)$ up to a certain equivalence that we will call weak conjugacy. Furthermore, we show that every periodic mapping class in a non-split metacyclic subgroup of $\mathrm{Mod}(S_g)$ is reducible.  We provide necessary and sufficient conditions under which a non-split metacyclic action on $S_g$ factors via a split metacyclic action. Finally, we provide a complete classification of the weak conjugacy classes of the finite non-split metacyclic subgroups of $\mathrm{Mod}(S_{10})$ and $\mathrm{Mod}(S_{11})$.
\end{abstract}

\section{Introduction}
\label{sec:intro}
Let $S_g$ be the closed orientable surface of genus $g \geq 0$, $\Homeo^+(S_g)$ be the group of orientation-preserving homeomorphisms of $S_g$, and let $\Mod(S_g)$ be the mapping class group of $S_g$. Given periodic elements $F,G \in \Mod(S_g)$ such that $\langle F, G \rangle$ is finite, a pair of conjugates $F',G'$ (of $F,G$ resp.) may (or may not) generate a subgroup isomorphic to $\langle F, G \rangle$ (see \cite{NKA}). A natural question that arises in this context is whether one can derive equivalent conditions under which $\langle F',G' \rangle \cong \langle F,G \rangle$. In this paper, we answer this question in the affirmative for the case when $\langle F,G \rangle$ is a metacyclic group~\cite{CEH} by 
deriving elementary number theoretic conditions for the same. Moreover, considering the fact that every cyclic subgroup of $\Mod(S_g)$ that lifts under a finite cyclic cover lifts to metacyclic group (see~\cite{MW}), our main result can also be viewed as a solution to the problem of liftability of periodic mapping classes under such covers up to conjugacy. Our main result (see Theorem~\ref{thm:main}) is a complete generalization of the main results in~\cite{DR} and~\cite{NKA}, where the analogous problem for split metacyclic groups was considered. The proof of this result applies the theory of group actions on surfaces~\cite{SK1,M1} and Thurston's orbifold theory~\cite[Chapter 13]{WT}. 

A finite \textit{metacyclic group $\M(u,n,r,k)$ of order $u\cdot n$, amalgam $r$ and twist factor $k$} admits a presentation of the form
$$\langle F,G \,\mid \, G^u = F^r, F^n = 1, G^{-1}FG = F^k \rangle.$$ A metacyclic group $\M(u,n,r,k)$ is said to be \textit{split} if $\M(u,n,r,k) \cong \Z_{n'}\rtimes\Z_{m'}$ for some positive integers $m',n'$. Consider a finite metacyclic subgroup $H = \langle F,G \rangle$ of $\Mod(S_g)$ as above. In view of the Nielsen realization theorem~\cite{SK,JN}, the $H$ acts on $S_g$ yielding a quotient orbifold $\O_H = S_g/H$, and $F$ has a Nielsen representative $\F \in \Homeo^+(S_g)$ of the same order. In a recent paper~\cite{ADDR}, a symplectic criterion has been derived for the liftability of mapping classes under regular cyclic covers. Moreover, it is well-known~\cite{AW} that for a $G \in \Mod(S_g)$ of finite order, $|G| \leq 4g+2$. These results motivate the following immediate corollary of our main result (from Section~\ref{sec:main}). 
\begin{cor*}
Let $H = \langle F,G \rangle$ be a metacyclic subgroup of $\Mod(S_{n(g-1)+1})$, where $|F| = n$ and $\langle \F \rangle$ acts freely on $S_{n(g-1)+1}$ inducing a regular cover $S_{n(g-1)+1} \to S_g$. 
\begin{enumerate}[(i)]
\item If $\O_H$ has genus zero, then $H$ is a split metacyclic group. 
\item If $n$ is prime and $(4g+2) \mid n-1$, then there exists a $\bar{G} \in \Mod(S_g)$ of order $4g+2$ that lifts to $G$.
\end{enumerate}
\end{cor*}

 In Section~\ref{sec:appl}, we provide several applications of our main theorem. A metacyclic group of the form $\M(2,2n,n,-1)$ is called a \textit{dicyclic} group denoted by $\Dic_n$. As a first application, we derive a realizable bound on $|H|$ when $H < \Mod(S_g)$ is a non-split metacyclic subgroup. 

\begin{prop*}
\label{prop1}
Let $H <  \Mod(S_g)$ be a non-split metacyclic subgroup. Then $|H| \leq 4g$ and this bound is realized when $H \cong \Dic_g$.
\end{prop*}

\noindent It may be noted here that a bound on the order of an arbitrary metacyclic subgroup of $\Mod(S_g)$ was derived in~\cite{AS}. As an immediate application of Proposition~\ref{prop1}, we obtain the following. 

\begin{cor*}
Let $H  < \Mod(S_g)$ be a finite metacyclic subgroup. 
\begin{enumerate}[(i)]
\item If $H$ is non-split, then any $F \in H$ is a reducible mapping class. 
\item If $H$ contains an irreducible mapping class $F$, then $H$ is split. Furthermore, if $H = \langle F,G\rangle$ and $\langle F \rangle \lhd H$ with $|F| = n$, then either $H \cong \Z_n \rtimes \Z_2$ or $H \cong \Z_n \rtimes \Z_3$.
\end{enumerate}
\end{cor*}

\noindent Furthermore, taking motivation from Proposition~\ref{prop1}, we give a combinatorial classification of dicyclic subgroups of $\Mod(S_g)$ (see Proposition~\ref{prop:dic}). 

Since a non-split metacyclic group is a quotient of a split metacyclic group with a cyclic subgroup, a natural question is when would a non-split metacyclic action on $S_g$ factor through a split metacyclic action. In this connection, we provide equivalent conditions under which a non-split metacyclic action on $S_g$ would lift under a regular cyclic cover to a split metacyclic action (see Proposition~\ref{prop:split}). By applying these conditions, we obtain the following corollary. 

\begin{cor*}
\label{cor2}
The actions on $S_g$ of the metacyclic groups $\Dic_n$, $\Dic_n \times \Z_m$, and $\Dic_n \times \Z_m \times \Z_p$, where $n$ is even integer and $m,p$ are odd integers with $\gcd(p,n)=1$, factor via split metacyclic actions.  
\end{cor*}

The geometric realizations of several split metacyclic actions on $S_g$ were described in~\cite{NKA}. But the realizations of non-split metacyclic group actions on $S_g$ are far more challenging as these groups are not realizable as isometry groups of $\mathbb{R}^3$. However, Corollary~\ref{cor2} (see also Proposition~\ref{prop:split}) enables us to realize the lifts of certain metacyclic actions under suitably chosen regular cyclic covers. In Section~\ref{sec:hyp_str}, we describe nontrivial geometric realizations of some finite split metacyclic actions that are lifts of non-split metacyclic actions on $S_{10}$ and $S_{11}$ under regular cyclic covers. Finally, in Section~\ref{sec:classify}, we provide classifications of the finite non-split metacyclic subgroups of $\Mod(S_{10})$ and $\Mod(S_{11})$ up to a certain equivalence that we call \textit{weak conjugacy} (see Section~\ref{sec:main}) which is central to the theory developed in this paper. 

\section{Preliminaries}
In this section, we introduce some basic notions from Thurston's orbifold theory~\cite[Chapter 13]{WT} and the theory of group actions on surfaces~\cite{SK1,M1} that are crucial to the theory we develop in this paper.
\label{sec:prelims}

\subsection{Fuchsian groups} Let $\Homeo^+(S_g)$ denote the group of orientation-preserving homeomorphisms of $S_g$, and let $H < \Homeo^+(S_g)$ be a finite group. A faithful and properly discontinuous $H$-action on $S_g$ induces a branched covering $S_g \to \mathcal{O}_H (:= S_g/H)$ with $\ell$ cone points $x_1,\ldots ,x_{\ell}$ on the quotient orbifold $\mathcal{O}_H \approx S_{g_0}$ (which we will call the \textit{corresponding orbifold}) of orders $n_1, \ldots ,n_{\ell}$, respectively. Then the orbifold fundamental group $\pi_1^{\orb}(\O_H)$ of $\O_H$ has a presentation given by
\begin{equation}
\label{eqn:orb-pres}
\left\langle \alpha_1,\beta_1,\dots,\alpha_{g_0},\beta_{g_0}, \xi_1,\dots,\xi_{\ell} \, \mid \, \xi_1^{n_1},\dots,\xi_\ell^{n_{\ell}},\,\prod_{j=1}^{\ell} \xi_j \prod_{i=1}^{g_0}[\alpha_i,\beta_i]\right\rangle.
\end{equation}
In classical parlance, $\pi_1^{\orb}(\O_H)$ is also known as a \textit{Fuchsian group}~\cite{SK1, M1} with \textit{signature}
$$\Gamma(\O_H) := (g_0;n_1,\ldots,n_{\ell}),$$ and the relation $\prod_{j=1}^{\ell} \xi_j \prod_{i=1}^{g_0}[\alpha_i,\beta_i]$ appearing in its presentation in (\ref{eqn:orb-pres}) is called the \textit{long relation}.
\noindent From Thurston's orbifold theory~\cite[Chapter 13]{WT}, we obtain exact sequence
\begin{equation}
\label{eq:surf_kern}
1 \rightarrow \pi_1(S_g) \rightarrow \pi_1^{\orb}(\O_H) \xrightarrow{\phi_H}  H \rightarrow 1,
\end{equation}
\noindent where $\phi_H$ is called the \textit{surface kernel} epimorphism. In this setting, we will require the following result due to Harvey~\cite{H1}, which gives a characterization of finite group action on surfaces.

\begin{theorem}
	\label{Harvey Condition} 
	A finite group $H$ acts faithfully on $S_g$ with $\Gamma(\O_H) = (g_0;n_1,\dots,n_{\ell})$ if and only if it satisfies the following two conditions: 
	\begin{enumerate}[(i)]
		\item $\displaystyle \frac{2g-2}{|H|}=2g_0-2+\sum_{i=1}^{\ell}\left(1-\frac{1}{n_i}\right)$, and 
		\item  there exists a surjective homomorphism $\phi_H:\pi_1^{\orb}(\O_H) \to H$ that preserves the orders of all torsion elements of $\pi_1^{\orb}(\O_H)$.
	\end{enumerate}
\end{theorem}

\subsection{Cyclic actions on surfaces } For $g \geq 1$, let $F \in \Mod(S_g)$ be of order $n$. By the Nielsen-Kerckhoff theorem~\cite{SK,JN}, $F$ is represented by a \textit{standard representative} $\F \in \Homeo^+(S_g)$ of the same order. We will refer to both $\F$ and the group it generates, interchangeably, as a \textit{$\Z_n$-action on $S_g$}.  Each cone point $x_i \in \O_{\langle \F \rangle}$ lifts to an orbit of size $n/n_i$ on $S_g$, and the local rotation induced by $\F$ around the points in each orbit is given by $2 \pi c_i^{-1}/n_i$, where $\gcd(c_i,n_i)=1$ and $c_i c_i^{-1} \equiv 1 \pmod{n_i}$. Further, by virtue of Theorem~\ref{Harvey Condition}, there exists an exact sequence:
\begin{equation*}
\label{eq:surf_kern}
1 \rightarrow \pi_1(S_g) \rightarrow \pi_1^{\orb}(\O_{\langle \F \rangle}) \xrightarrow{\phi_{\langle \F \rangle}}  \langle \F \rangle \rightarrow 1, 
\end{equation*}
\noindent where $\phi_{\langle \F \rangle} (\xi _i) = \F^{(n/n_i)c_i}$, for $1 \leq i \leq \ell$. We will now introduce a tuple of integers that encodes the conjugacy class of a $\Z_n$-action on $S_g$.  

\begin{definition}\label{defn:data_set}
	A \textit{cyclic data set of degree $n$} is a tuple
	$$
	D = (n,g_0, d; (c_1,n_1),\ldots, (c_{\ell},n_{\ell})),
	$$
	where $n\geq 2$, $g_0 \geq 0$, and $0 \leq d \leq n-1$ are integers, and each $c_i \in \Z_{n_i}^\times$ such that:
	\begin{enumerate}[(i)]
		\item $d > 0$ if and only if $\ell = 0$ and $\gcd(d,n) = 1$, whenever $d >0$,
		\item each $n_i\mid n$,
		\item $\lcm(n_1,\ldots \widehat{n_i}, \ldots,n_{\ell}) = N$, for $1 \leq i \leq \ell$, where $N = n$ if $g_0 = 0$,  and
		\item $\displaystyle \sum_{j=1}^{\ell} \frac{n}{n_j}c_j \equiv 0\pmod{n}$.
	\end{enumerate}
	The number $g$ determined by the Riemann-Hurwitz equation
	\begin{equation}\label{eqn:riemann_hurwitz}
	\frac{2-2g}{n} = 2-2g_0 + \sum_{j=1}^{\ell} \left(\frac{1}{n_j} - 1 \right) 
	\end{equation}
	is called the \textit{genus} of the data set, denoted by $g(D)$.
\end{definition}

\noindent Note that quantity $d$ (in Definition~\ref{defn:data_set}) will be non-zero if and only if $D$ represents a free rotation of $S_g$ by $2\pi d/n$, in which case, $D$ will take the form $(n,g_0,d;)$. We will not include $d$ in the notation of a data set, whenever $d = 0$.  

By the Nielsen-Kerckhoff theorem, the quotient map $\Homeo^+(S_g)\to \Mod(S_g)$ induces a  correspondence between the conjugacy classes of finite-order maps in $\Homeo^+(S_g)$ and the conjugacy classes of finite-order mapping classes in $\Mod(S_g)$. This leads us to the following lemma primarily due to Nielsen~\cite{JN1} which allows us to use data sets to describe the conjugacy classes of cyclic actions on $S_g$ (see also~\cite[Theorem 3.8]{KP} and~\cite{H1}). 

\begin{lemma}\label{prop:ds-action}
	For $g \geq 1$ and $n \geq 2$, data sets of degree $n$ and genus $g$ correspond to conjugacy classes of $\Z_n$-actions on $S_g$. 
\end{lemma}

\noindent We will denote the data set corresponding to the conjugacy class of a periodic mapping class $F$ by $D_F$. For compactness of notation, we also write a data set $D$ (as in Definition~\ref{defn:data_set}) as
$$D = (n,g_0,d; ((d_1,m_1),\alpha_1),\ldots,((d_{\ell'},m_{\ell'}),\alpha_{\ell'})),$$
where $(d_i,m_i)$ are the distinct pairs in the multiset $S = \{(c_1,n_1),\ldots,(c_{\ell},n_{\ell})\}$, and the $\alpha_i$ denote the multiplicity of the pair $(d_i,m_i)$ in the multiset $S$. Further, we note that every cone point $[x] \in \mathcal{O}_{\langle \F \rangle}$ corresponds to a unique pair in the multiset $S$ appearing in $D_F$, which we denote by $\P_x := (c_{x},n_{x})$. 

Given $u \in \Z_m^{\times}$ and $\G \in H < \Homeo^+(S_g)$ be of order $m$, let $\mathbb{F}_{\G} (u,m)$ denote the set of fixed points of $\G$ with induced rotation angle $2\pi u/m$. Let $C_H(\G)$ be the centralizer of a $\G \in H$ and $\sim$ denote the conjugation relation between any two elements in $H$. We conclude this subsection by stating the following result from the theory of Riemann surfaces~\cite{TB}, which we will use in the proof of our main theorem.

\begin{lemma}
	\label{lem:riem_surf_fix}
	Let $H < \Homeo^+(S_g)$ with $\Gamma(\O_H)  = (g_0; n_1,\ldots,n_{\ell})$, and let $\G \in H$ be of order $m$. Then for $u \in \Z_m ^{\times}$, we have
	\[\displaystyle |\mathbb{F}_{\G} (u,m)| = |C_H(\G)| \cdot \sum_{\displaystyle \substack{1 \leq i \leq \ell \\ m \mid n_i \\ \G \sim \phi_H(\xi_i)^{n_i u/m}}} \frac{1}{n_i}.\]
\end{lemma}

\subsection{Metacyclic actions on surfaces} Given integers $u,n \geq 2$, $r \mid n$ and $k \in \Z_n^{\times}$ such that $k^u \equiv 1 \pmod{n}$, a \textit{finite metacyclic action of order $un$} (written as $u \cdot n$) on $S_g$ is a tuple $(H, (\G,\F))$, where 
$H < \Homeo^+(S_g)$, and $$H = \langle \F,\G \, \mid \, \F^r = \G^u, \F^n = 1, \G^{-1}\F\G = \F^k \rangle \cong \M(u,n,r,k).$$ Metacyclic groups have been completely classified in~\cite{CEH}. We will call the multiplicative class $k$ the \textit{twist factor} and $r$ be the \textit{amalgam} of the metacyclic action $(H, (\G,\F))$. In the presentation above, if we further assume that $r=n$, then $\M(u,n,r,k) \cong \Z_n \rtimes_{k} \Z_u$, which is also known as a \textit{split metacyclic group}. Since split metacyclic actions on surfaces have been analyzed in~\cite{NKA}, we will focus our attention on the finite non-split metacyclic subgroups of $\Mod(S_g)$. 

Since $\langle \F \rangle \lhd H$, $\G$ would induce a $\bar{\G} \in \Homeo^+(\O_{\langle \F \rangle})$ (see~\cite{TWT}) that restricts to an order-preserving bijection on the set of cone points in $\O_{\langle \F \rangle}$. We will call $\bar{\G}$, \textit{the induced automorphism on $\O_{\langle \F \rangle}$ by $\G$,} and we formalize this notion in the following definition. 

\begin{definition}
	\label{defn:ind_auto}
	Let $H < \Homeo^+(S_g)$  be a finite cyclic group with $|H| = n$. We say a $\bar{\G} \in \Homeo^+(\O_H)$  is an \textit{automorphism of $\O_H$} if for $[x],[y] \in \O_H$, $k \in \Z_n^{\times}$ and $\bar{\G}([x]) = [y]$, we have: 
	\begin{enumerate}[(i)]
		\item $n_x = n_y$, and 
		\item $ c_x =k c_y$. 
	\end{enumerate} 
	We denote the group of automorphisms of $\O_H$ by $\Aut(\O_H)$.
\end{definition}

\noindent We will require the following technical lemma that gives some basic properties of the induced automorphism. 

\begin{lemma}
	\label{lem:ind_aut}
	Let $\G,\F \in {\Homeo}^+(S_g)$ be maps of orders $m,n$, respectively, such that $\G^{-1} \F \G = \F^k$, and let $H = \langle \F\rangle$. Then:
	\begin{enumerate}[(i)]
		\item  $\G$ induces a $\bar{\G}\in \Aut(\O_H)$ such that $$\O_H/\langle \bar{\G}\rangle = S_g/\langle \F,\G\rangle,$$ 
		\item $|\bar{\G}| \text{ divides } |\G|$, and 
		\item $|\bar{\G}| < m$ if and only if $\F^{r}=\G^u$, for some $0< r<n$ and $0< u< m$.
	\end{enumerate}
\end{lemma}

\section{Main theorem}
\label{sec:main}
In this section, we establish the main result of the paper by deriving equivalent conditions under which torsion elements $F,G \in \Mod(S_g)$ can have conjugates $F',G' \in \Mod(S_g)$ such that $\langle F',G' \rangle \cong \M(u,|F'|,r,k)$. We begin by introducing a weaker notion of conjugacy that arises very naturally in this context. 
\begin{defn}
	\label{defn:weak_conj}
	Two finite metacyclic actions $(H_1, (\G_1, \F_1))$ and $(H_2, (\G_2, \F_2))$ of order $u \cdot n$, amalgam $r$ and twist factor $k$ are said to be \textit{weakly conjugate} if there exists an isomorphism, $\psi: \pi_1^{\orb}(\O_{H_1}) \cong \pi_1^{\orb}(\O_{H_2})$ and an isomorphism $\chi : H_1 \to H_2$ such that the following conditions hold. 
	\begin{enumerate}[(i)]
		\item $\chi((\G_1,\F_1)) = (\G_2,\F_2)$.
		\item For $i = 1,2$, let $\phi_{H_i}: \pi_1^{\orb}(\O_{H_i}) \to H_i$ be surface kernel epimorphisms. Then $(\chi \circ \phi_{H_1})(g) = (\phi_{H_2} \circ \psi)(g), \text{ whenever } g \in \pi_1^{\orb}(\O_{H_1})$ is of finite order.
		\item The pair $(\G_1, \F_1)$ is conjugate (component-wise) to the pair $(\G_2, \F_2)$ in $\Homeo^+(S_g).$
	\end{enumerate}
	\noindent The notion of weak conjugacy defines an equivalence relation on metacyclic actions on $S_g$ and the equivalence classes thus obtained will be called \textit{weak conjugacy classes}. 
\end{defn}

\begin{rem}
	\label{rem:ext_weak_conj}
In view of the Nielsen-Kerckhoff theorem, the notion of weak conjugacy (from Definition~\ref{defn:weak_conj}) naturally extends to an analogous notion in $\Mod(S_g)$ via the natural association 
	$$(\langle \F, \G \rangle,(\G,\F)) \leftrightarrow  (\langle F, G \rangle,(G,F)).$$ 
\end{rem}

\noindent For brevity, we will introduce the following notation. 
\begin{defn}
	Let $F,G \in \Mod(S_g)$ be of finite order with $|F|= n$. Then for some $k \in \Z_{n}^\times \setminus \{1 \}$, we say (in symbols) that $\lb F,G \rb_{u,r,k} = 1$ if there exists conjugates $F',G'$ (of $F,G$ resp.) such that $\langle F',G' \rangle \cong \M(u,n,r,k)$.
\end{defn}

We will now define an abstract tuple of integers that will encode the weak conjugacy class of a finite metacyclic actions on $S_g$.
\begin{defn}
\label{defn:meta_cyc_dataset}
	A \textit{metacyclic data set of degree $u \cdot n$, twist factor $k$, amalgam $r$ and genus $g \geq 2$} is a tuple
$$\D = ((u \cdot n,r,k),g_0;[(c_{11},n_{11}),(c_{12},n_{12}),n_1], \ldots , [(c_{\ell 1},n_{\ell 1}),(c_{\ell 2},n_{\ell 2}),n_{\ell} ]),$$
	where $u,n \geq 2$, the $n_{ij}$ are positive integers for $1 \leq i \leq \ell, ~ 1 \leq j \leq 2$, the $c_{ij} \in \Z_{n_{ij}}$, $r \mid n$ and $k \in \Z_n^{\times}$ such that $k^u \equiv 1 \pmod{n}$ and there exists a $w \in \Z$, satisfying the following conditions.
	
		\begin{enumerate}[(i)]
		\item $\displaystyle	\frac{2g-2}{un} = 2g_0 - 2 + \sum_{i=1}^{\ell} \left( 1 - \frac{1}{n_i} \right) .$
		\item \begin{enumerate}
			\item For each $i,j$, $n_{i1} \mid  \dfrac{un}{r} := m$, $n_{i2} \mid n$, either $\text{gcd}(c_{ij},n_{ij}) = 1$ or $c_{ij}=0$, and $c_{ij}=0$ if and only if $n_{ij} = 1$.
			\item For each $i$, $n_i = s_i$, where $s_i$ is the least positive integer satisfying the following  conditions for some $t_i \in \mathbb{N}$: 
			\begin{enumerate}
			\item $c_{i1}\frac{m}{n_{i1}} s_i \equiv t_i u \pmod{m}.$
			\item $c_{i2}\frac{n}{n_{i2}} (k^{c_{i1}\frac{m}{n_{i1}}(s_i -1)}+ \dots + k^{c_{i1}\frac{m}{n_{i1}}}+1) \equiv -t_i r \pmod{n}.$
			\end{enumerate}

		\end{enumerate}
	\item $\displaystyle \sum_{i = 1}^{\ell} c_{i1} \frac{m}{n_{i1}} \equiv wu \pmod{m}$. 
	\item Defining  $\displaystyle A := \sum_{i=1}^{\ell} c_{i2} \frac{n}{n_{i2}} \prod_{s = i+1}^{\ell} k^{c_{s1} \frac{m}{n_{s1}}}$ and $d:= \gcd(n,k-1)$, we have
	$$A \equiv \begin{cases}
	-wr \pmod{n}, & \text{if } g_0=0, \text{ and} \\
	d\theta -wr \pmod{n}, \text{ for } \theta \in \Z_n, & \text{if } g_0 \geq 1.
	\end{cases}$$
	\item If  $g_0 = 0$, there exists $(p_1, \ldots , p_{\ell v}), (q_1, \ldots, q_{\ell v}) \in \Z^{\ell v}$, $v \in \mathbb{N}$, and $a,b \in \Z$ such that the following conditions hold. 
	\begin{enumerate}
	\item $ \displaystyle \sum_{i'=1}^{\ell v} p_{i'} c_{i1} \frac{m}{n_{i1}} \equiv 1+au \pmod{m}$ and $$\sum_{i'=1}^{\ell v} c_{i2} \frac{n}{n_{i2}} \left(\sum_{s=1}^{p_{i'}} k^{c_{i1} \frac{m}{n_{i1}}(p_{i'} -s)}\right) \left(\prod_{t' = i'+1}^{\ell v} k^{p_{t'}c_{t1} \frac{m}{n_{t1}}}\right) \equiv -ar \pmod{n}.$$
	\item $\displaystyle \sum_{i'=1}^{\ell v} q_{i'} c_{i1} \frac{m}{n_{i1}} \equiv bu \pmod{m}$ and 
	$$\sum_{i'=1}^{\ell v} c_{i2} \frac{n}{n_{i2}} \left(\sum_{s=1}^{q_{i'}} k^{c_{i1} \frac{m}{n_{i1}}(q_{i'} -s)}\right) \left(\prod_{t' = i'+1}^{\ell v} k^{q_{t'}c_{t1} \frac{m}{n_{t1}}}\right) \equiv 1 -br \pmod{n}, \text{ where } $$ 
$$i \equiv
	\begin{cases} 
	i' \pmod{\ell}, & \text{if } i' \not \equiv 0 \pmod{\ell}, \\
	\ell & \text{otherwise,}\\
	 \end{cases} 
\,\,t \equiv
	\begin{cases} 
	t' \pmod{\ell}, & \text{if } t'  \not \equiv 0 \pmod{\ell}, \text{ and}\\
	\ell, & \text{otherwise.}\\
	 \end{cases}$$
	\end{enumerate}
	\item If  $g_0 = 1$, there exists $(p_1, \ldots , p_{\ell v}), (q_1, \ldots, q_{\ell v}) \in \Z^{\ell v}$, $m',n', a, b \in \Z$, and $v \in \mathbb{N}$ such that $m' \mid m$ and $n' \mid n$, satisfying the following conditions.
	\begin{enumerate}
	\item $ \displaystyle \sum_{i'=1}^{\ell v} p_{i'} c_{i1} \frac{m}{n_{i1}} \equiv m' + au \pmod{m}$ and $$\sum_{i'=1}^{\ell v} c_{i2} \frac{n}{n_{i2}} \left(\sum_{s=1}^{p_{i'}} k^{c_{i1} \frac{m}{n_{i1}}(p_{i'} -s)}\right) \left(\prod_{t' = i'+1}^{\ell v} k^{p_{t'}c_{t1} \frac{m}{n_{t1}}}\right) \equiv -ar \pmod{n}.$$
	\item $\displaystyle \sum_{i'=1}^{\ell v} q_{i'} c_{i1} \frac{m}{n_{i1}} \equiv bu \pmod{m} \text{ and }$ 
	$$\sum_{i'=1}^{\ell v} c_{i2} \frac{n}{n_{i2}} \left(\sum_{s=1}^{q_{i'}} k^{c_{i1} \frac{m}{n_{i1}}(q_{i'} -s)}\right) \left(\prod_{t' = i'+1}^{\ell v} k^{q_{t'}c_{t1} \frac{m}{n_{t1}}}\right) \equiv n' -br \pmod{n},$$ 
	where $$i \equiv 
	\begin{cases} 
	i' \pmod{\ell}, & \text{if } i'  \not \equiv 0 \pmod{\ell}, \\
	\ell, & \text{otherwise,}
	 \end{cases}
	\,\, t \equiv \begin{cases} 
	t' \pmod{\ell} & \text{if } t'  \not \equiv 0 \pmod{\ell}, \text{ and}\\
	\ell & \text{otherwise.}
	 \end{cases}$$
	\item $A \equiv -\beta k^{\alpha} + \beta -wr \pmod{n}$ for some non-negative integers $\alpha,$ $\beta$, where $$\displaystyle \lcm \left( \frac{m}{m'},\frac{m}{\gcd(m,\alpha)} \right) = m \text{ and } \lcm \left( \frac{n}{n'},\frac{n}{\gcd(n,\beta)} \right) = n.$$ Furthermore, we set $\alpha = 1$, when $m'=0$, and $\beta =1$, when $n'=0$.
	
	\end{enumerate}
	\end{enumerate}	 		
\end{defn}

\begin{prop}
\label{prop:main}
For integers $n,u,g,r \geq 2$ such that $r \mid n$ and $k \in \Z_n^{\times}$, the metacyclic data sets of degree $u \cdot n$ with twist factor $k$, amalgam $r$, and genus $g$ correspond to the weak conjugacy classes of $\M(u,n,r,k)$-actions on $S_g$.
\end{prop}

\begin{proof}
Let $\D$ be a metacyclic data set of degree $u \cdot n$ with twist factor $k$, amalgam $r$ and genus $g$ (as in Definition~\ref{defn:meta_cyc_dataset} above). We need to show that $\D$ represents the weak conjugacy class of a $\M(u,n,r,k)$-action on $S_g$ represented by $(H,(\G,\F))$, where $H = \langle \F,\G \rangle$. To see this, we first show the existence of an epimorphism $\phi_H : \pi_1^{orb} (\O_H) \rightarrow H$ that preserves the orders of the torsion elements. Let $H$ and $\pi_1^{orb}(\O_H)$ have presentations given by
\begin{gather*} 
\langle \F,\G\, \mid \,\F^n = 1,  \G^u = \F^r, \G^{-1}\F\G = \F^{k} \rangle \cong \M(u,n,r,k)\text{ and} \\
\langle \alpha_1, \beta_1 , \cdots , \alpha_{g_0} , \beta_{g_0}, \xi_1 , \cdots , \xi_{\ell} \, \mid \, \xi_1^{n_1} = \cdots = \xi_{\ell}^{n_\ell} = \prod_{j=1}^{\ell} \xi_j \prod_{i=1}^{g_0}[\alpha_i , \beta_i] = 1 \rangle, 
\end{gather*} respectively.
	
	We consider the map
	$$\displaystyle \xi_i \xmapsto{\phi_H}  \G^{c_{i1} \frac{m}{n_{i1}}} \F^{c_{i2} \frac{n}{n_{i2}}}, \text{ for } 1 \leq i \leq \ell,$$
	where $m := \dfrac{un}{r}$. Then condition (ii) of Definition~\ref{defn:meta_cyc_dataset} would imply that $\phi_H$ is a map which is order-preserving on torsion elements. For clarity, we break the argument for the surjectivity of $\phi_H$ into the following three cases. 

\textit{Case 1: $g_0 = 0$}. Then it follows from conditions (iii)-(iv) that $\phi_H$ satisfies the long relation $\prod_{i=1}^{\ell} \xi_i = 1$. Moreover, the surjectivity of $\phi_H$ follows from condition (v). 
	
\textit{Case 2: $g_0 \geq 2$.} In this case, $\pi_1^{orb} (\O_H)$ has additional hyperbolic generators (viewing them as isometries of the hyperbolic plane), namely the $\alpha_i$ and the $\beta_i$. Extending $\phi_H$ by mapping $\alpha_1 \xmapsto{\phi_H} \G, \beta_1 \xmapsto{\phi_H} \F$ yields an epimorphism. Moreover, by carefully choosing the images of the $\alpha_i$ and the $\beta_i$, for $i \geq 2$, conditions (iii)-(iv) would ensure that the long relation $\prod_{j=1}^{\ell} \xi_j \prod_{i=1}^{g_0}[\alpha_i ,\beta_i] = 1$ is satisfied. 
	
\textit{Case 3: $g_0 = 1$.} In this case, $\pi_1^{orb} (\O_H)$ would have two additional hyperbolic generators, namely the $\alpha_1$ and the $\beta_1$. We extend $\phi_H$ by defining $\alpha_1 \xmapsto{\phi_H} \G^{\alpha}$ and $\beta_1 \xmapsto{\phi_H} \F^{\beta}$. We then apply conditions (iii),(iv) and (vi) to obtain the desired epimorphism.
	
It remains to be shown that $\D$ determines $\F,\G \in \Homeo^{+}(S_g)$ up to conjugacy. But by carefully applying Lemma~\ref{lem:riem_surf_fix}, we see that 

\begin{gather*}\displaystyle D_F = (n,g_1;((v_{ij}^{-1},t_i),\frac{t_i |\mathbb{f}_{\F^{\frac{n}{t_i}}}(v_{ij},t_i)|}{n}) : v_{ij} \in \mathbb{Z}_{t_i}^{\times}, \, t_i \mid n),
\end{gather*}
where $$\displaystyle |\mathbb{f}_{\F^{\frac{n}{t_i}}}(v_{ij},t_i)|= |\mathbb{F}_{\F^{\frac{n}{t_i}}}(v_{ij},t_i)| - \sum_{\substack{t_{i'} \in \mathbb{N} \\ t_{i'} \neq t_i \\ t_i \mid t_{i'} \mid n}} \sum_{\substack{ (v_{i'j'},t_{i'}) = 1 \\  v_{ij} \equiv v_{i'j'} (\text{mod} \, t_i)}} |\mathbb{f}_{\F^{\frac{n}{t_{i'}}}}(v_{i'j'},t_{i'})|$$ and $g_1$ is determined by Riemann-Hurwitz equation, and

\begin{gather*}\displaystyle D_G = (m,g_2;((u_{ij}^{-1},m_i),\frac{m_i |\mathbb{f}_{\G^{\frac{m}{m_i}}}(u_{ij},m_i)|}{m}) : u_{ij} \in \mathbb{Z}_{m_i}^{\times}, \, m_i \mid m),
\end{gather*}
where $$\displaystyle |\mathbb{f}_{\G^{\frac{m}{m_i}}}(u_{ij},m_i)|= |\mathbb{F}_{\G^{\frac{m}{m_i}}}(u_{ij},m_i)| - \sum_{\substack{m_{i'} \in \mathbb{N} \\ m_{i'} \neq m_i \\ m_i \mid m_{i'} \mid m}} \sum_{\substack{ (u_{i'j'},m_{i'}) = 1 \\  u_{ij} \equiv u_{i'j'} (\text{mod} \, m_i)}} |\mathbb{f}_{\G^{\frac{m}{m_{i'}}}}(u_{i'j'},m_{i'})|$$ and $g_2$ is determined by Riemann-Hurwitz equation.
	
	Conversely, consider the weak conjugacy class of $\M(u,n,r,k)$-actions on $S_g$ represented by $(H, (\G,\F))$, where $H = \langle \F,\G \rangle$. By Theorem~\ref{Harvey Condition}, there exists a surjective homomorphism $$\phi_H : \pi_1^{orb}(\O_H) \to H : \displaystyle \xi_i \xmapsto{\phi_H} \G^{c_{i1} \frac{m}{n_{i1}}} \F^{c_{i2} \frac{n}{n_{i2}}}, \text{ for } 1 \leq i \leq \ell, $$
which is order-preserving on the torsion elements. This yields a metacyclic data set $\D$ of degree $u \cdot n$ with twist factor $k$, amalgam $r$ and genus $g$ as in Definition~\ref{defn:meta_cyc_dataset}. By Theorem~\ref{Harvey Condition}, $\D$ satisfies condition (i) of Definition~\ref{defn:meta_cyc_dataset}. Moreover, condition (ii) follows from the fact that $\phi_H$ is order-preserving on torsion elements. Furthermore, conditions (iii)-(iv) follow from the long relation satisfied by $\pi_1^{orb}(\O_H)$, and condition (v)-(vi) are implied by the surjectivity of $\phi_H$. Thus, we obtain the metacyclic data set $\D$ of degree $u \cdot n$ with twist factor $k$, amalgam $r$ and genus $g$, and our assertion follows.
\end{proof}

\noindent We denote the data sets $D_F$ and $D_G$ (representing the cyclic factors of $H$) derived from the metacyclic data set $\D$ appearing in the proof of Proposition~\ref{prop:main} by $\D_1$ and $\D_2$, respectively. Thus, our main theorem will now follow from Proposition~\ref{prop:main}.

\begin{theorem}[Main theorem]
\label{thm:main}
Let $F,G \in \Mod(S_g)$ be of orders $n,m$, respectively. Then $\lb F,G \rb_{u,r,k} = 1$ if and only if there exists a metacyclic data set $\D$ of degree $u \cdot n$, twist factor $k$, amalgam $r$, and genus $g$ such that $\D_1 = D_F$ and $\D_2 = D_G$.
\end{theorem}

\subsection{Liftability of torsion under finite cyclic covers}
\label{sec:lift}

From the viewpoint of liftability, a metacyclic group $\langle \F,\G \rangle$ acts on $S_g$ if and only if there exists $\bar{\G} \in \Aut(\O_{\langle \F \rangle})$ that lifts under the branched cover $S_g \to \O_{\langle \F \rangle}$ to $\G$. From Birman-Hilden theory~\cite{MW}, this is equivalent to requiring the existence of a short exact sequence:
$$1 \to \langle \F \rangle \to \langle \F, \G \rangle \to \langle \bar{\G} \rangle \to 1.$$

Let $S_{h,b}$ denote the closed oriented surface of genus $h$ with $b$ marked points. Given a branched cover $p : S_g \to S_{h,b}$, let $\LMod_p(S_{h,b})$ (resp. $\SMod_p(S_g)$) denote the liftable (resp. symmetric) mapping class group of $p$. Our main theorem can now be equivalently stated as follows.
	\begin{theorem}[Main theorem-Alternative version]
		\label{thm:lift}
		Let $p : S_g \rightarrow S_{h,b}$ be an $n$-sheeted cover with deck transformation group $\langle \F \rangle \cong \Z_n$. Then $\bar{G} \in \LMod_p(S_{h,b})$ lifts to a $G \in \SMod_p(S_g)$ if and only if there exists a metacyclic data set $\mathcal{D}$ of degree $u \cdot n$, twist factor $k$, amalgam $r$, and genus $g$ such that $\mathcal{D}_1 = D_F$ and $\D_2 = D_{G}$. 
	\end{theorem}
\noindent Thus, the main theorem provides necessary and sufficient conditions under which periodic elements of mapping class groups will lift under finite cyclic covers. 

\begin{rem}
\label{rem:g_bar}
Given a metacyclic data set 
$$\mathcal{D} = ((u \cdot n,r,k),g_0;[(c_{11},n_{11}),(c_{12},n_{12}),n_1], \ldots , [(c_{\ell 1},n_{\ell 1}),(c_{\ell 2},n_{\ell 2}),n_{\ell} ]),$$ encoding the weak conjugacy class represented by $(\langle \F,\G \rangle, (\G,\F))$, it follows from Theorem~\ref{thm:main} that
	$$D_{\bar{G}} = (u,g_0;(c'_{11},n'_{11}), \ldots, (c'_{\ell 1},n'_{\ell 1})),$$
where $\displaystyle n'_{i1} = \frac{u}{\gcd(c_{i1} \frac{m}{n_{i1}},u)}$, and $\displaystyle c'_{i1}\frac{u}{n'_{i1}} \equiv c_{i1} \frac{m}{n_{i1}} \pmod{u}$. 
\end{rem}

\noindent This leads us to the following corollary, which is an application of Theorem~\ref{thm:lift}. 
	\begin{cor}
		\label{cor:free}
Let $p:S_{n(g-1)+1} \rightarrow S_g$ be an $n$-sheeted regular cover with deck transformation group $\langle \F \rangle \cong \Z_n$. Suppose that there exists a $\bar{G} \in \LMod_p(S_g)$ of finite order with $\O_{\langle \bar{\G}\rangle} \approx S_0$ that lifts to a $G \in \SMod_p(S_{n(g-1)+1})$. Then $H = \langle F,G \rangle$ is a split metacyclic group.	
	\end{cor}
	
\begin{proof}
	From Theorem~\ref{thm:lift}, we have a metacyclic data set 
	$$\mathcal{D} = ((u \cdot n,r,k),0;[(c_{11},n_{11}),(c_{12},n_{12}),n_1], \ldots , [(c_{\ell 1},n_{\ell 1}),(c_{\ell 2},n_{\ell 2}),n_{\ell} ])$$
	of degree $u \cdot n$ with twist factor $k$, amalgam $r$, and genus $n(g-1)+1$.
	Following the notation in the proof of Proposition~\ref{prop:main}, let $\phi_H(\xi_i) = \G^{\gamma_i} \F^{\delta_i}$, where $\gamma_i = c_{i1}m/n_{i1}$, $\delta_i = c_{i2}n/n_{i2}$, and $\xi_i \in \pi_1^{orb}(\O_H)$ is the generator enclosing the cone point of order $n_i$. Now, as $\F$ generates a free action on $S_{n(g-1)+1}$, we have that $\langle \G^{\gamma_i} \F^{\delta_i} \rangle \cap \langle \F \rangle = \{id\}.$
Hence, it follows that $\langle G^{\gamma_i} F^{\delta_i},F \rangle$ is a split metacyclic group for all $i$, and consequently, $\langle G^{\gamma_i},F \rangle$ is a split metacyclic group for all $i$. 

Now, we claim that $\langle G^{\gamma_1}, \hdots , G^{\gamma_{\ell}},F \rangle$ is a split metacyclic group. We establish this claim by inducting on $\ell$. From the preceding argument, the statement holds for $\ell =1.$ For $\ell = 2,$ we have to show that $ \langle G^{\gamma_1},  G^{\gamma_{2}},F \rangle$ is a split metacyclic group. We can write $ \langle G^{\gamma_1},  G^{\gamma_{2}},F \rangle =  \langle G',F \rangle$, where $\langle G' \rangle = \langle G^{\gcd(\gamma_1,\gamma_2)}\rangle = \langle G^{\gamma_1},  G^{\gamma_{2}} \rangle$. Suppose we assume on the contrary that $(G')^a = F^b$, for some $a \in \Z$ and $b \in \Z_n$ with $b \neq 0$. Then $\langle (G')^a \rangle \subseteq \langle G^{\gamma_1},  G^{\gamma_{2}} \rangle$, and so we have that $\langle (G')^{at} \rangle \subseteq \langle G^{\gamma_1} \rangle \text{ or } \langle  G^{\gamma_{2}} \rangle$, for some $t$ such that $(G')^{at} \neq 1$. Hence, it follows that $G^{\gamma_1 t_1} = (G')^{at} = F^{bt} \text{ or } G^{\gamma_2 t_2} = (G')^{at} = F^{bt}$, for some $t_1,t_2 \in \Z$, which contradicts the fact that $\langle G^{\gamma_i},F \rangle$ is a split metacyclic group. Therefore, our claim holds true for $\ell = 2$. 

Suppose we assume that our claim holds for $\ell -1$. By similar arguments (as above), we have that $ \langle G^{\gamma_1}, \hdots , G^{\gamma_{\ell}},F \rangle = \langle G'', G^{\gamma_{\ell}},F \rangle$, where $\langle G'' , F \rangle $ is split metacyclic group. So, it immediately follows from the case $\ell = 2$ that our claim holds for $\ell$.  Since $\phi_H$ is surjective, we have $ \langle G^{\gamma_1}F^{\delta_1}, \hdots , G^{\gamma_{\ell}}F^{\delta_{\ell}} \rangle = H$, and hence it follows that $H$ is a split metacyclic group.
	
\end{proof}

As mentioned earlier, the bound on the order of a periodic mapping class $G \in \Mod(S_g)$ is $4g+2$ (see~\cite{AW}) which is realized by the action $D_G = (4g+2,0;(1,2),(1,2g+1),(2g-1,4g+2))$. This inspires the following corollary. 

\begin{cor}
	Let $p:S_{n(g-1)+1} \rightarrow S_g$ be a finite n-sheeted regular cover with deck transformation group $\langle \F \rangle\ \cong \Z_n$. If $n$ is prime and $(4g+2) \mid (n-1)$, then there exists a $\bar{G} \in \LMod_p(S_g)$ with $|\bar{G}| = 4g+2$. 
\end{cor}
\begin{proof}
Since $(4g+2) \mid (n-1)$, there exists a $k \in \Z_n^{\times}$ such that $|k| = 4g+2$. Let $G \in \SMod_p(S_{n(g-1)+1})$ be a lift of $\bar{G}$. Now from Proposition~\ref{prop:main}, it can be easily seen that the metacyclic data set
$\D = ((n \cdot 4g+2,n,k),0;[(1,2),(1,n),2],[(1,2g+1),(n-k^2,n),2g+1],[(2g-1,4g+2),(0,1),4g+2])$ represents the weak conjugacy class of $(\langle \F,\G \rangle,\G,\F)$ with $\D_1 = D_F= (n,g,1;)$, $\D_2 = D_G = (4g+2,\frac{(n-1)(g-1)}{4g+2};(1,2),(1,2g+1),(2g-1,4g+2))$ and $D_{\bar{G}} = (4g+2,0;(1,2),(1,2g+1),(2g-1,4g+2)).$ Hence, our assertion follows.
\end{proof}

\section{Applications}
\label{sec:appl}

\subsection{Bound on the order of a non-split metacyclic action}

\label{subsec:bound}
In this subsection, we derive a realizable bound for the order of a non-split metacyclic subgroup of $\Mod(S_g)$. We will need the following technical lemma concerning metacyclic groups.
\begin{lemma}
	\label{lem:prime}
Let $H = \langle F,G \,\mid \, G^u = F^r, F^n = 1, G^{-1}FG = F^k \rangle.$ Suppose that there exists $x,y \in H$ such that $H = \langle x,y \rangle$ and at least one of $x$ or $y$ is of prime order.  Then $H$ is a split metacyclic group. 
\end{lemma}
\begin{proof}
Let $x = G^{i_1}F^{j_1}$ and $y = G^{i_2}F^{j_2}$ and let us assume without loss of generality that $x$ is of prime order $p$. Suppose we assume on the contrary that $H$ is a non-split metacyclic group. Then either $G^{i_1}F^{j_1} = F^{\alpha}$, for some $\alpha$, or $G^{i_1}F^{j_1} = (G')^{\beta}$, for some $\beta$, where $G'$ is chosen so that $H = \langle F , G' \rangle$. This implies that $H = \langle G^{i_1}F^{j_1}, G^{i_2}F^{j_2} \rangle = \langle F^{\alpha}, G^{i_2}F^{j_2} \rangle$. Since $F^{\alpha}$ is of prime order and $\langle F^{\alpha} \rangle \lhd H$, $H$ must be split metacyclic group, which contradicts our assumption. A similar argument works for the case when $G^{i_1}F^{j_1} = (G')^{\beta}$. Hence, it follows that $H$ is a split metacyclic group.
\end{proof}

We call $\Dic_{n} := \M(2,2n,n,-1)$ the \textit{dicyclic group} of order $4n$. We will now derive a realizable bound on the order of a finite non-split metacyclic subgroup of $\Mod(S_g)$.

\begin{prop}
	\label{prop:bound}
Suppose that $H < \Mod(S_g)$ is a finite non-split metacyclic group. Then $|H| \leq 4g$ and this bound is realized when $g$ is even and $H \cong \Dic_{g}$.
\end{prop}
\begin{proof}
	We will show that if $H < \Mod(S_g)$ such that $|H| > 4g$, then $H$ cannot be a non-split metacyclic group. If $\Gamma(\O_H) = (g_0;n_1,n_2, \dots, n_{\ell})$, then $H$ satisfies the Riemann-Hurwitz equation: $$\displaystyle \frac{2g-2}{|H|}=2g_0-2+\sum_{i=1}^{\ell}\left(1-\frac{1}{n_i}\right).$$ When $|H| > 4g$, we have 
\begin{equation} 
\label{eqn:4g}
\displaystyle 2g_0-2+\sum_{i=1}^{\ell}\left(1-\frac{1}{n_i}\right)=\frac{2g-2}{|H|}<\frac{2g-2}{4g}=\frac{g-1}{2g}<\frac{1}{2},
\end{equation}
from which it follows that $g_0 = 0$ and $\ell = 3 \text{ or } 4$. 

From Lemma \ref{lem:prime}, if $g_0 = 0$, $ \ell = 3$, and there is a cone point of prime order, then $H$ cannot be a non-split metacyclic group. So, by (\ref{eqn:4g}), when $H$ is a non-split metacyclic group with $|H| > 4g$, the possible signatures for $\pi_1^{orb}(\O_H)$ are $(0;2,2,3,3),(0;2,2,3,4),(0;2,2,3,5),(0;2,2,2,n),$ $(0;4,4,n),(0;4,6,6),$ $(0;4,6,8), (0;4,6,9)$, and $(0;4,6,10)$,  where $n<2g$. We will now show that none of these signatures will arise from a non-split metacyclic action. 

Assume that $H$ is a metacyclic group. Then $H = \langle F,G\rangle$, where $\F \in \Homeo^+{(S_g)}$ with $\mathcal{O}_{\langle \F \rangle} \approx S_{h,b}$ and $\bar{\G} \in \mathrm{Aut}_k(\O_{\langle \F \rangle})$. From Remark~\ref{rem:g_bar}, we have that $\Gamma(\O_{\langle \bar{\G} \rangle}) = (0;m_1,m_2, \hdots, m_{\ell})$, where $m_i \mid n_i.$ First, we will consider the case when $\ell =4$.
If $h \neq 0$, from Proposition~\ref{prop:ds-action},  we get $\Gamma(\O_{\langle \bar{\G} \rangle})$  equals either $(0;2,2,3,3)$ or $(0;2,2,2,2)$. Thus, $\F$ either generates a free action or $\Gamma(\O_{\langle \F \rangle}) = (1;\frac{n}{2})$. But, by Proposition~\ref{prop:ds-action} and Corollary~\ref{cor:free}, we can see that neither of these possibilities occur when $H$ is a non-split metacyclic group. If $h = 0$, then $\Gamma(\O_{\langle \bar{\G} \rangle}) = (0;u,u)$, where $u \in \{2,3\}$. Again, from Theorem~\ref{thm:main}, we see that $\Gamma(\O_{\langle \F \rangle})$ equals one of $(0;2,2,2,2,2,2), (0;3,3,3,3),(0;3,3,4,4),$ $(0;2,2,2,3,3),$ $(0;3,3,5,5)$, $(0;2,2,n,n)$, or $(0;2,2,2,2,n/2)$. Hence, either $H$ is a split metacyclic group or $|H| \leq 4g$. This completes our argument for $\ell = 4.$

Now, for $\ell = 3$, by similar arguments as above, we can conclude that $\Gamma(\O_{\langle \F \rangle}) = (0;n,n,n,n)$ (or $(0;4,4,4,4,4,4)$). By Lemma~\ref{lem:ind_aut} (iii) and Theorem~\ref{thm:main}, we see that $|\G| = |\bar{\G}|$, and so $H$ is a split metacyclic group. 

For the realization of the bound, when $H \cong \Dic_g$ and $g$ is even, we can see that data set 
$$\D = ((2\cdot 2g,g,-1),0;[(1,4),(0,1),4],[(1,4),(1,2g),4],[(0,1),(2g-1,2g),2g])$$ represents the weak conjugacy class of $(H,(\G,\F))$. 
	
\end{proof}

\noindent An immediate consequence of Proposition~\ref{prop:bound} is the following.

\begin{cor}
\label{cor:irr_order}
Suppose that $H < \Mod(S_g)$ is a finite non-split metacyclic group. Then there exists no irreducible periodic mapping class in $H$. 
\end{cor}
\begin{proof}
Since $|H| \leq 4g$ and $H$ is non-split, we have $|F| \leq 2g$ for any $F \in H$. Our assertion now follows from the fact that the order of any irreducible periodic mapping class is at least $2g+1$.
\end{proof}

\noindent Corollary~\ref{cor:irr_order} further yields the following.

\begin{cor}
Suppose that $H = \langle F, G \rangle < \Mod(S_g)$ is a finite non-split metacyclic group. 
\begin{enumerate}[(i)]
\item If $g=2$, then $|F| \leq 2g$, $|G| \leq 2g$, and $|\bar{G}| \leq g$. Moreover, these bounds are realized when $H \cong Q_8$.
\item If  $g > 2$, then $|F| \leq 2g$, $|G| \leq 2g-2$, and $|\bar{G}| \leq g-1$. Moreover, the bound on $|F|$ is realized when $H \cong \mathrm{Dic}_{g}$, where $g$ is even, while the bounds on $|G|$ and $|\bar{G}|$ are realized when $H \cong Q_8 \times \Z_{\frac{g-1}{2}}$, where $g \equiv 3 \pmod{4}$.
\end{enumerate}
\end{cor}

\begin{proof}
From Corollary~\ref{cor:irr_order}, we have that $|F|,|G| \leq 2g$. Also, as $|\bar{G}| < |G|$, from Lemma~\ref{lem:ind_aut}, we have $|\bar{G}| \leq g$. Hence, the assertion in (i) follows immediately from Proposition~\ref{prop:bound}. 
	
Furthermore, by Proposition~\ref{prop:bound}, the bound $4g$ on $|H|$ is realized when $H \cong \Dic_g$ and $g$ is even. It is apparent that $|F| = 2g$ in $H$, which realizes the required bound in the first part of (ii). Moreover, from the proof of Proposition~\ref{prop:bound} it is apparent that $\Dic_g$ will not realize the bounds for $|G|$ and $|\bar{G}|$. However, it can be easily seen that the bounds on $|G|$ and $|\bar{G}|$ are realized when $H \cong Q_8 \times \Z_{\frac{g-1}{2}}$ with the weak conjugacy class $(H,(G,F))$ represented by the metacyclic data set 
$$(((g-1)\cdot 4,2,-1),1;[(0,1),(1,2),2]).$$
\end{proof}

\noindent We conclude this subsection with the following direct consequence of Corollary~\ref{cor:irr_order}.

\begin{cor}
	Let $p:S_g \rightarrow S_{0,3}$ be a finite n-sheeted cover with deck transformation group $\langle \F \rangle \cong \Z_n$. If $\bar{G} \in \LMod_p(S_{0,3})$ lifts to a $G \in \SMod_p(S_g)$. Then $H = \langle F, G \rangle$ is a split metacyclic group. Furthermore, either $H \cong \Z_n \rtimes_{k} \Z_2$ or $H \cong \Z_n \rtimes_{k} \Z_3$.
\end{cor}

\subsection{Dicyclic subgroups of $\Mod(S_g)$}
\label{subsec:dicyc}
In Proposition~\ref{prop:bound}, we saw the significance of dicyclic groups as bound-realizing metacyclic subgroups of $\Mod(S_g)$. This motivates a separate analysis of dicyclic actions, which is precisely what we undertake in this subsection. We recall that a dicyclic group of order $4n$ is given by $\Dic_{n} := \M(2,2n,n,-1)$. We will call a metacyclic data set of degree $2\cdot 2n$, amalgam $n$ and twist factor $-1$, a \textit{dicyclic data set}. Note that a dicyclic group is a non-split metacyclic group if and only if $n$ is even. Thus, throughout this subsection, $n$ will be assumed to be even. The following is an immediate consequence of Proposition~\ref{prop:main}. 

\begin{cor}
For $g \geq 2$ and $n \geq 3$, dicyclic data sets of degree $2 \cdot 2n$ and genus $g$ correspond to the weak conjugacy classes of $\Dic_{n}$-actions on $S_g$.
\end{cor}

\begin{rem}
	\label{rem:G_fix_point}
Let $H = \Dic_n = \langle F,G \rangle < \Mod(S_g)$. Then $\bar{\G}$ cannot fix a regular point in orbifold $S_g/ \langle \F \rangle$. To see this, suppose we assume on the contrary that $\bar{G}([x]) = [x],$ where $[x]$ is a regular point in $\O_{\langle \F \rangle}.$ Then $\mathrm{Stab}_{\langle G,F \rangle}(x) = \langle G^2 \rangle = \langle F^n \rangle$, which implies that $[x]$ is an order 2 cone point in $\O_{\langle \F \rangle}$, thereby yielding a contradiction. 
\end{rem}

\noindent The following proposition provides an alternative characterization of a $\Dic_{n}$-action on $S_g$ .

\begin{prop}
	\label{prop:dic}
	Let $F \in \Mod(S_g)$ be of order $2n$. Then there exists a $G \in \Mod(S_g)$ of order $4$ such that $\langle F, G \rangle \cong \Dic_{n}$ if and only if $D_F$ has the form 
	\[(2n,g_0,d; ((c_1,n_1),(-c_1,n_1), \ldots, (c_{s},n_{s}),(-c_{s},n_{s})) \tag{*}\] satisfying the following conditions.
	\begin{enumerate}[(i)]
		\item When $g_0$ is even, there exists an $i$ such that $(c_i,n_i) = (-c_i,n_i)= (1,2).$
		\item When $g_0$ is odd, at least one of the following statements hold true. 
		\begin{enumerate}
			\item There exists $i,j$ with $i \neq j$ such that  $(\pm c_i, n_i) = (\pm c_j,n_j) = (1,2)$.
			\item $g_0 \geq 3$ and $	\displaystyle \sum_{i=1}^{s} c_{i} \frac{2n}{n_{i}} \equiv 2a \pmod{2n}$ for some $a \in \Z$.
			\item $g_0 =1$ and $	\displaystyle \sum_{i=1}^{s} c_{i} \frac{2n}{n_{i}} \equiv 2 \pmod{2n}$.
			\item $g_0 =1$, $\displaystyle \sum_{i=1}^{s} c_{i} \frac{2n}{n_{i}} \equiv 2a \pmod{2n}$ for some $a \in \Z$, and $\lcm(n_1, \hdots ,n_s) = 2n$.
		\end{enumerate}		
		
 	\end{enumerate}
\end{prop}

\begin{proof}
	Suppose that $D_F$ has the form (*). Then $\O_{\langle \F \rangle}$ is an orbifold of genus $g_0$ with $2s$ cone points $[x_1],[y_1],\ldots, [x_s],[y_s]$, where $\P_{x_i} = (c_i,n_i)$ and 
	$\P_{y_i} = (-c_i,n_i)$, for $1 \leq i \leq s$. If $D_F$ satisfies condition (i), then we may assume without loss of generality that $(c_1,n_1)=(-c_1,n_1)=(1,2)$. Then up to conjugacy, let $\bar{\G} \in \Aut( \O_{\langle \F \rangle})$ be an involution such that $\bar{\G}([x_i])= [y_i]$, for $2 \leq i \leq s$, $\bar{\G}([x_1]) = [x_1]$, and $\bar{\G}([y_1]) = [y_1]$. To prove our assertion, it would suffice to show the existence of an involution $\G \in \Homeo^+(S_g)$ that induces $\bar{\G}$. This amounts to showing that there exists a metacyclic data set $\D$ of degree $2 \cdot 2n$ with amalgam $n$ and twist factor $-1$ encoding the weak conjugacy class $(H,(\G,\F))$ so that $D_G$ has degree $4$. 
	
	Consider the tuple
	\begin{gather*} 
	\D = ((2 \cdot 2n,n,-1),g_0/2;[(1,4),(0,1),4], [(3,4),(c',n') ,4],\\  [(0,1),(c_{2},n_{2}),n_2], \hdots , [(0,1),(c_{s},n_{s}),n_s]), 
	\end{gather*}
	where $c' \frac{2n}{n'} \equiv -\sum_{i=2}^{s} c_i \frac{2n}{n_i} \pmod{2n}$.
	\noindent It follows immediately that $\D$ satisfies conditions (i)-(iv) of Definition~\ref{defn:meta_cyc_dataset}.  By taking $v =1$, we may choose $(p_1, \hdots , p_{s+1}) = (1,0, \hdots , 0)$ to conclude that $\D$ also satisfies either condition (v)(a) or (vi)(a), based on the choice of $g_0$. If $g_0 = 0$, we have that $\text{lcm}(n_1, \hdots , n_s) = 2n$, from which condition (v)(b) follows. If $g_0 \neq 0,$ then by carefully defining $\phi_H$ (as in Proposition~\ref{prop:main}) on the hyperbolic elements of $\pi_1^{orb}(\O_H)$, our claim is true. Thus, it follows that $\D$ is a metacyclic data set.
	
	If $D_F$ satisfies condition (ii)(a), then by a similar argument as above, we obtain the metacyclic data set 
	\begin{gather*} 
	\D = ((2 \cdot 2n,n,-1),(g_0+1)/2;[(1,4),(0,1),4],[(1,4),(0,1),4],[(1,4),(1,2n),4], \\
	[(1,4),(c'',n''),4],  [(0,1),(c_{3},n_{3}),n_3], \hdots , [(0,1),(c_{s},n_{s}),n_s]), 
	\end{gather*}
	where $c'' \frac{2n}{n''} \equiv 1-\sum_{i=3}^{s} c_i \frac{2n}{n_i} \pmod{2n}$.
	Suppose that $D_F$ satisfies conditions (ii)(b)-(d). Then again by an analogous argument as above, we obtain the metacyclic data set
	\begin{gather*} 
	\D = ((2 \cdot 2n,n,-1),(g_0+1)/2; [(0,1),(c_{1},n_{1}),n_1], \hdots , [(0,1),(c_{s},n_{s}),n_s]).
	\end{gather*}
	 Further, a direct application of Theorem~\ref{thm:main} would show that $\D$ indeed encodes the weak conjugacy represented by $(H,(\G,\F))$, as desired. 
	
	The converse follows immediately from Remark~\ref{rem:ext_weak_conj}, Remark~\ref{rem:G_fix_point} and Proposition~\ref{prop:main}.
\end{proof}

\subsection{Liftability of non-split metacylic actions under regular cyclic covers}
\label{subsec:liftable}

Considering the fact that every non-split metacyclic group is a quotient of a split metacyclic group, a natural question arises is when a given metacyclic action on $S_g$ factor via a split metacyclic action. In other words, when does a metacyclic action on $S_g$ lift under a regular cover to a split metacyclic action. In the following proposition (which follows directly from Theorem~\ref{thm:main}), we provide an equivalent condition for the liftability of a metacyclic action under a regular cyclic cover. 

\begin{prop}
	\label{prop:split}
	Let $p_{\nu} : S_{\nu(g-1)+1} \rightarrow S_g$ be a regular cyclic cover, and let $H  = \langle F, G \rangle < \Mod(S_g)$ be a finite non-split metacyclic group such that $H \cong \M(u,n,r,k)$ and the weak conjugacy class $(H,(G,F))$ encoded by the data set $$\mathcal{D} = ((u \cdot n,r,k),g_0;[(c_{11},n_{11}),(c_{12},n_{12}),n_1], \ldots , [(c_{\ell 1},n_{\ell 1}),(c_{\ell 2},n_{\ell 2}),n_{\ell} ]).$$ Then $H$ lifts under $p_{\nu}$ to a split metacyclic group $\tilde{H} = \langle \tilde{F}, \tilde{G} \rangle < \Mod(S_{v(g-1)+1})$ such that $\tilde{H} \cong \M(\nu u, n, n, k) \cong \Z_n \rtimes_k \Z_{\nu u}$ if and only if
	\begin{enumerate}[(i)]
	\item $\nu = n/r$ and
	\item the weak conjugacy class $(\tilde{H},(\tilde{G},\tilde{F}))$ is encoded by the data set
	$$\tilde{\D} = ((m \cdot n,n,k),g_0;[(c'_{11},n'_{11}),(c'_{12},n'_{12}),n_1], \ldots , [(c'_{\ell 1},n'_{\ell 1}),(c'_{\ell 2},n'_{\ell 2}),n_{\ell} ]),$$
	where $m = u \frac{n}{r}$,  $c'_{i1}\frac{m}{n'_{i1}} \equiv c_{i1}\frac{m}{n_{i1}} + a_iu \pmod{m}$ and $c'_{i2}\frac{n}{n'_{i2}} \equiv c_{i2}\frac{n}{n_{i2}} - a_ir \pmod{n},$ for some $a_i \in \Z$.
	\end{enumerate}
	\end{prop}
	
\noindent An immediate consequence of Proposition~\ref{prop:split} is the following.
\begin{cor}
\label{cor:lift_meta_groups}
The actions on $S_g$ of the metacyclic groups $\Dic_n$, $\Dic_n \times \Z_m$, and $\Dic_n \times \Z_m \times \Z_p$, where $n$ is even and $m,p$ are odd with $\gcd(p,n)=1$, factor via split metacyclic actions.  
\end{cor}

\noindent Proposition~\ref{prop:split} and Corollary~\ref{cor:lift_meta_groups} motivate the following conjecture. 

\begin{conj}
Every non-split metacyclic action on $S_g$ lifts under a suitably chosen finite regular cyclic cover to a split metacyclic action.
\end{conj}
  		
 \section{Geometric realizations of the lifts of non-split metacyclic actions}
 \label{sec:hyp_str}
In this section, we use Corollary~\ref{cor:lift_meta_groups} to provide explicit geometric realizations of the lifts of some non-split metacyclic actions on $S_{10}$ and $S_{11}$. These realizations implicitly assume the theory developed in~\cite{NKA,PKS}. The associated weak conjugacy classes of these actions are represented by the metacyclic data sets listed in Tables~\ref{tab:meta_dsets1}-\ref{tab:meta_dsets2} in Section~\ref{sec:classify}.

\pagebreak
 
 \begin{figure}[htbp]
 	\centering
 	\labellist
 	\tiny
 	\pinlabel $\G$ at 420 422
 	\pinlabel $\frac{\pi}{2}$ at 425 388
 	\pinlabel $(1,6)$ at 285 570
 	\pinlabel $(1,6)$ at 285 480
 	\pinlabel $(1,6)$ at 290 73
 	\pinlabel $(1,6)$ at 290 -15 
 	\pinlabel $(5,6)$ at 472 271
 	\pinlabel $(5,6)$ at 599 271
 	\pinlabel $(5,6)$ at -28 271
 	\pinlabel $(5,6)$ at 99 271
 	\endlabellist
 	\includegraphics[width=45ex]{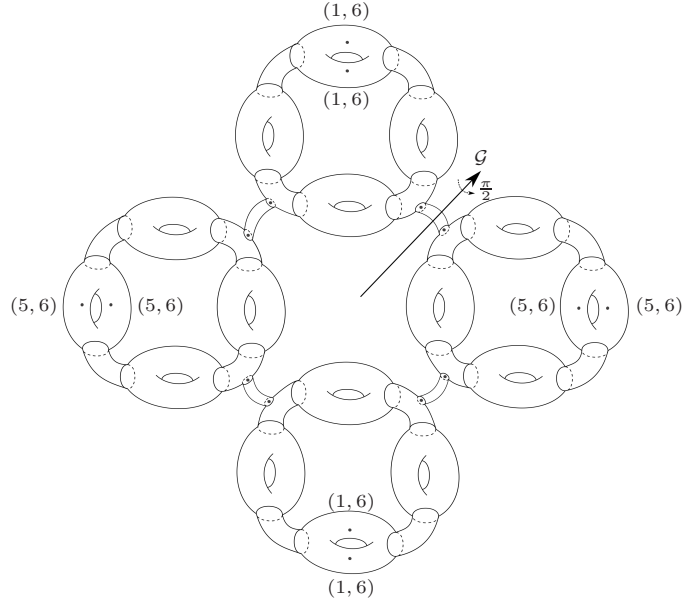}
 	\caption{The realization of a $H=\mathbb{Z}_{12} \rtimes_{-1} \mathbb{Z}_4$-action on $S_{21}$ which is the lift of a $\Dic_6$-action on $S_{11}$ under the regular cyclic cover $p_2$. Here, $H = \langle \F,\G \rangle$, where $D_{G} = (4,6,1;)$ and $D_F = (12,1;((1,6),2),((5,6),2))$. Note that the $\G$ maps each orbit  of the $\langle \F \rangle$-action of size 2 with local rotation angle $2\pi/6$ to an orbit with local rotation angle $10\pi/6$ (and vice versa).}
 	\label{fig:Z12_Z4_S21}
 \end{figure}
 
   \begin{figure}[htbp]
   	\centering
   	\labellist
   	\tiny
   	\pinlabel $(19,20)$ at 175 100
   	\pinlabel $(19,20)$ at 175 55
   	\pinlabel $(1,20)$ at 330 100
   	\pinlabel $(1,20)$ at 330 55
   	\pinlabel $(1,4)$ at -22 75
   	\pinlabel $(1,4)$ at 545 80
   	\pinlabel $(1,2)$ at 250 110
   	\pinlabel $(1,2)$ at 250 45
   	\endlabellist
   	\includegraphics[width=52ex]{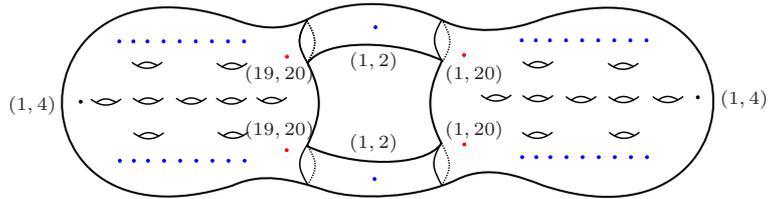}
   	\caption{The realization of a $H=\mathbb{Z}_{20} \rtimes_{-1} \mathbb{Z}_4$-action on $S_{19}$ which is the lift of a $\Dic_{10}$-action on $S_{10}$ under the regular cyclic cover $p_2$. Here, $H = \langle \F,\G \rangle$, where $D_{G} = (4,0;((1,4),2),((1,2),19))$ and $D_F = (20,0;((1,20),2),((19,20),2))$. Note that the four fixed points of $\F$ (marked in red) form an orbit of size 4 under the $\langle \G \rangle$-action where each fixed point with local rotation $2\pi/20$ is mapped to fixed point with local rotation $38\pi/20 $ (and vice versa). The point marked in blue are distinct size $2$ orbits of the $\langle \G \rangle$-action, while the points marked in black are the fixed points of $\G$.}
   	\label{fig:Dic10S19}
   \end{figure}
 
  \begin{figure}[H]
  	\centering
  	\labellist
  	\tiny
  	\pinlabel $(3,4)$ at 250 90
  	\pinlabel $(3,4)$ at 250 50
  	\pinlabel $(1,4)$ at 360 92
  	\pinlabel $(1,4)$ at 360 52
  	\pinlabel $(1,4)$ at 172 72
  	\pinlabel $(1,4)$ at 445 75
   	\pinlabel $(1,2)$ at 305 55
   	\pinlabel $(1,2)$ at 305 85
   	\pinlabel $(1,2)$ at 220 90
   	\pinlabel $(1,2)$ at 220 50
   	\pinlabel $(1,2)$ at 390 90
   	\pinlabel $(1,2)$ at 390 53
  	\endlabellist
  	\includegraphics[width=58ex]{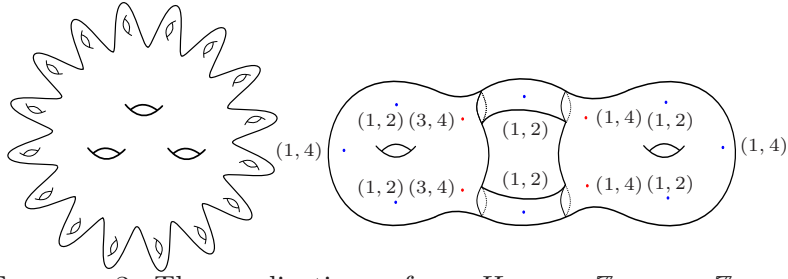}
  	\caption{The realization of a $H=\mathbb{Z}_{4} \rtimes_{-1} \mathbb{Z}_4$-action on $S_{19}$ which is the lift of a $Q_8$-action on $S_{10}$ under the regular cyclic cover $p_2$. Here, $H = \langle \F,\G \rangle$, where $D_F = (4,4;((1,4),2),((3,4),2))$ and $D_{G} = (4,4;((1,4),2),((1,2),3))$. Note that the $H$-action on $S_{19}$ cyclically permutes the genera in the petals of the subfigure on the left. This $H$-action is induced by an analogous action of an $H' = \langle \F', \G' \rangle \cong \mathbb{Z}_{4} \rtimes_{-1} \mathbb{Z}_4$ on $S_3$ (shown in the subfigure on the right) with $D_{F'} = (4,0;((1,4),2),((3,4),2))$ and $D_{G'} = (4,0;((1,4),2),((1,2),3))$. Note that the four fixed points of $\F'$ (marked in red) form an orbit of size 4 under the $\langle \G' \rangle$-action where each fixed point with local rotation $2\pi/4$ is mapped to fixed point with local rotation $6\pi/4 $ (and vice versa). The remaining fixed (and orbit) points of the $\langle \G' \rangle$-action are marked in blue.} 
  	\label{fig:Z4_Z4_S19}
  \end{figure}
  
  \section{Classification of the weak conjugacy classes in \\ $\Mod(S_{10})$ and $\Mod(S_{11})$}
 \label{sec:classify}
 
 In this section, we will use Theorem~\ref{thm:main} to classify the weak conjugacy classes of the non-split metacyclic subgroups of $\Mod(S_{10})$ and $\Mod(S_{11})$. For brevity, we will further assume the following equivalence of the metacyclic data sets (i.e. the weak conjugacy classes). 
 \begin{defn}
 	\label{defn:eq_data_sets}
 	Two metacyclic data sets 
 	\begin{gather*}
 	\D=((u \cdot n,r,k),g_0;[(c_{11},n_{11}),(c_{12},n_{12}),n_1], \ldots , [(c_{\ell 1},n_{\ell 1}),(c_{\ell 2},n_{\ell 2}),n_{\ell} ]) \\ \text{ and } 
 	\\ \D'=((u \cdot n,r,k),g_0;[(c'_{11},n'_{11}),(c'_{12},n'_{12}),n'_1], \ldots , [(c'_{\ell 1},n'_{\ell 1}),(c'_{\ell 2},n'_{\ell 2}),n'_{\ell} ])
 	\end{gather*} 
 	\noindent are said to be \textit{equivalent} if for each tuple $[(c'_{i1},n'_{i1}),(c'_{i2},n'_{i2}),n'_i]$, there exists a unique tuple $[(c_{j1},n_{j1}),(c_{j2},n_{j2}),n_j]$ satisfying the following conditions:
 	\begin{enumerate}[(i)]
 		\item $n'_i = n_j,$
 		\item $c'_{i1}\frac{m}{n'_{i1}} \equiv c_{j1}\frac{m}{n_{j1}} + au \pmod{m}$, where $m = u \frac{n}{r}$, and
 		\item $c'_{i2} \frac{n}{n'_{i2}} \equiv c_{j2} \frac{n}{n_{j2}} k^{a_i} + b_i (k^{c_{j1} \frac{m}{n_{j1}}} -1) - ar \pmod{n}$ for some $a_i,b_i,a \in \mathbb{Z}.$
 	\end{enumerate}
 \end{defn}
 \noindent Note that equivalent data sets $\D$ and $\D'$ as in Definition~\ref{defn:eq_data_sets} satisfy 
 $\D_i' = \D_i$, for $i = 1,2$. We will now provide a classification of the weak conjugacy classes of finite non-split metacyclic subgroups of $\Mod(S_{10})$ and $\Mod(S_{11})$ (up to this equivalence) in Tables~\ref{tab:meta_dsets1} and~\ref{tab:meta_dsets2}, respectively. 
 
  \begin{landscape}
  	\begin{table}[htbp]
  		\small
  		\begin{center}
  			\resizebox{21cm}{!}
  			{\begin{tabular}{|c|c|c|}
  					\hline
  					\textbf{Group} & \textbf{Weak conjugacy classes in $\Mod(S_{10})$} & \textbf{Cyclic factors $[D_G;D_F]$}\\ \hline
  					\multirow{5}{*}{$\M(2,4,2,-1)$} & $((2 \cdot 4,2,-1),0;[(0,1),(1,2),2],[(1,4),(0,1),4]_3,[(0,1),(1,4),4],[(1,4),(3,4),4])$ & $[(4,1;(1,4),(3,4),((1,2),6));(4,0;((1,4),3),((3,4),3)((1,2),2))]$\\ \cline{2-3} 
  					& \makecell{$((2 \cdot 4,2,-1),0;[(0,1),(1,2),2],[(1,4),(0,1),4],[(0,1),(1,4),4]_3,[(3,4),(1,4),4])$} & $[(4,0;((1,4),3),((3,4),3)((1,2),2));(4,1;(1,4),(3,4),((1,2),6))]$\\ \cline{2-3}
  					& \makecell{$((2 \cdot 4,2,-1),0;[(0,1),(1,2),2],[(1,4),(0,1),4],[(0,1),(3,4),4],[(1,4),(1,4),4]_3)$} & $[(4,1;(1,4),(3,4),((1,2),6));(4,1;(1,4),(3,4),((1,2),6))]$\\ \cline{2-3}
  					& \makecell{$((2 \cdot 4,2,-1),0;[(0,1),(1,2),2]_4,[(1,4),(0,1),4],[(0,1),(1,4),4],[(3,4),(1,4),4])$} & $[(4,0;(1,4),(3,4),((1,2),10));(4,0;(1,4),(3,4),((1,2),10))]$\\ \cline{2-3}
  					& \makecell{$((2 \cdot 4,2,-1),1;[(1,4),(0,1),4],[(0,1),(1,4),4],[(1,4),(1,4),4])$} & $[(4,2;(1,4),(3,4),((1,2),2));(4,2;(1,4),(3,4),((1,2),2))]$\\
  					\hline
  					\multirow{2}{*}{$\M(2,8,4,-1)$} & $((2 \cdot 8,4,-1),0;[(1,4),(0,1),4],[(1,4),(1,8),4],[(0,1),(1,4),4],[(0,1),(1,8),8])$ & $[(4,1;(1,4),(3,4),((1,2),6));(8,0;(1,8),(7,8),(1,4),(3,4),((1,2),2))]$\\ \cline{2-3}
  					& \makecell{$((2 \cdot 8,4,-1),0;[(1,4),(0,1),4],[(1,4),(7,8),4],[(0,1),(1,4),4],[(0,1),(3,8),8])$} & $[(4,1;(1,4),(3,4),((1,2),6));(8,0;(3,8),(5,8),(1,4),(3,4),((1,2),2))]$\\
  					\hline
  					\multirow{2}{*}{$\M(2,12,6,7)$} & $((2 \cdot 12,6,7),0;[(1,4),(1,12),12],[(3,4),(1,3),12],[(0,1),(1,12),12])$ & $[(4,2;(1,4),(3,4),((1,2),2));(12,0;(1,12),(7,12),((1,6),2))]$\\ \cline{2-3}
  					& \makecell{$((2 \cdot 12,6,7),0;[(1,4),(1,6),12],[(3,4),(5,12),12],[(0,1),(5,12),12])$} & $[(4,2;(1,4),(3,4),((1,2),2));(12,0;(5,12),(11,12),((5,6),2))]$\\
  					\hline
  					\multirow{4}{*}{$\M(2,20,10,-1)$} & $((2 \cdot 20,10,-1),0;[(1,4),(0,1),4],[(1,4),(9,20),4],[(0,1),(1,20),20])$ & $[(4,0;(1,4),(3,4),((1,2),10));(20,0;(1,20),(19,20),((1,2),2))]$\\ \cline{2-3} 
  					& \makecell{$((2 \cdot 20,10,-1),0;[(1,4),(0,1),4],[(1,4),(7,20),4],[(0,1),(3,20),20])$} & $[(4,0;(1,4),(3,4),((1,2),10));(20,0;(3,20),(17,20),((1,2),2))]$\\ \cline{2-3}
  					& \makecell{$((2 \cdot 20,10,-1),0;[(1,4),(0,1),4],[(1,4),(3,20),4],[(0,1),(7,20),20])$} & $[(4,0;(1,4),(3,4),((1,2),10));(20,0;(7,20),(13,20),((1,2),2))]$\\ \cline{2-3}
  					& \makecell{$((2 \cdot 20,10,-1),0;[(1,4),(0,1),4],[(1,4),(1,20),4],[(0,1),(9,20),20])$} & $[(4,0;(1,4),(3,4),((1,2),10));(20,0;(9,20),(11,20),((1,2),2))]$\\
  					\hline
  					
  				\end{tabular}}
  				\caption{The weak conjugacy classes of finite non-split metacyclic subgroups of $\Mod(S_{10})$.(*The suffix refers to the multiplicity of the tuple in the non-split metacyclic data set.)}
  				\label{tab:meta_dsets1}
  			\end{center}
  		\end{table}

  		\begin{table}[H]
  			\small
  			\begin{center}
  				\resizebox{21cm}{!}
  				{\begin{tabular}{|c|c|c|}
  						\hline
  						\textbf{Group} & \textbf{Weak conjugacy classes in $\Mod(S_{11})$} & \textbf{Cyclic factors $[D_G;D_F]$}\\ \hline
  						$\M(2,12,6,-1)$ & $((2 \cdot 12,6,-1),1;[(0,1),(1,6),6])$ & $[(4,3;((1,2),2));(12,1;(1,6),(5,6))]$\\ \hline
  						\multirow{8}{*}{$\M(4,8,4,-1)$} & $((4 \cdot 8,4,-1),0;[(1,8),(0,1),8],[(7,8),(7,8),8],[(0,1),(1,8),8])$ & $[(8,0;(1,8),(5,8),(1,4),((3,4),2),(1,2));(8,0;((1,8),2),((7,8),2)((1,2),2))]$\\ \cline{2-3} 
  						& \makecell{$((4 \cdot 8,4,-1),0;[(1,8),(1,8),8],[(7,8),(0,1),8],[(0,1),(1,8),8])$} & $[(8,0;(3,8),(7,8),((1,4),2),(3,4),(1,2));(8,0;((1,8),2),((7,8),2)((1,2),2))]$\\ \cline{2-3}
  						& \makecell{$((4 \cdot 8,4,-1),0;[(1,8),(0,1),8],[(7,8),(5,8),8],[(0,1),(3,8),8])$} & $[(8,0;(1,8),(5,8),(1,4),((3,4),2),(1,2));(8,0;((3,8),2),((5,8),2)((1,2),2))]$\\ \cline{2-3}
  						& \makecell{$((4 \cdot 8,4,-1),0;[(1,8),(3,8),8],[(7,8),(0,1),8],[(0,1),(3,8),8])$} & $[(8,0;(3,8),(7,8),((1,4),2),(3,4),(1,2));(8,0;((3,8),2),((5,8),2)((1,2),2))]$\\ \cline{2-3}
  						& \makecell{$((4 \cdot 8,4,-1),0;[(1,8),(1,8),8],[(5,8),(0,1),8],[(1,4),(1,8),8])$} & $[(8,0;(1,8),(5,8),((1,4),3),(1,2));(8,1;(1,4),(3,4)((1,2),2))]$\\ \cline{2-3}
  						& \makecell{$((4 \cdot 8,4,-1),0;[(1,8),(3,8),8],[(5,8),(0,1),8],[(1,4),(3,8),8])$} & $[(8,0;(1,8),(5,8),((1,4),3),(1,2));(8,1;(1,4),(3,4)((1,2),2))]$\\ \cline{2-3}
  						& \makecell{$((4 \cdot 8,4,-1),0;[(3,8),(1,8),8],[(3,8),(0,1),8],[(1,4),(1,8),8])$} & $[(8,0;(3,8),(7,8),((3,4),3),(1,2));(8,1;(1,4),(3,4)((1,2),2))]$\\ \cline{2-3}
  						& \makecell{$((4 \cdot 8,4,-1),0;[(3,8),(3,8),8],[(3,8),(0,1),8],[(1,4),(3,8),8])$} & $[(8,0;(3,8),(7,8),((3,4),3),(1,2));(8,1;(1,4),(3,4)((1,2),2))]$\\
  						\hline
  						$\M(2,20,10,11)$ & $((2 \cdot 20,10,11),1;[(0,1),(1,2),2])$ & $[(4,1;((1,2),10));(20,1;((1,2),2))]$\\ \hline
  					\end{tabular}}
  					\caption{The weak conjugacy classes of finite non-split metacyclic subgroups of $\Mod(S_{11})$ other than quaternions.(*The suffix refers to the multiplicity of the tuple in the non-split metacyclic data set.)}
  					\label{tab:meta_dsets2}
  				\end{center}
  			\end{table}
  		\end{landscape}

 \section*{Acknowledgements}
 The second author was partly supported by a UGC-JRF fellowship. The authors would also like to thank Dr. Neeraj Kumar Dhanwani for some helpful discussions.

 \bibliographystyle{plain} 
 \bibliography{metacyclic_action}

\end{document}